\documentclass[11pt,reqno]{amsproc}

\usepackage{amsmath, amsthm, amscd, amsfonts, amssymb, graphicx, color}
\usepackage[bookmarksnumbered, plainpages]{hyperref}
\usepackage{graphicx}
\usepackage{epsfig}
\newtheorem{thm}{Theorem}[section]

\newtheorem{lem}[thm]{Lemma}
\newtheorem{prop}[thm]{Proposition}

\theoremstyle{definition}
\newtheorem{defn}[thm]{Definition}

\theoremstyle{remark}
\newtheorem{rem}[thm]{Remark}
\numberwithin{equation}{section}

 \numberwithin{equation}{section}

\title{Module cohomological properties of the Fourier algebra of an inverse semigroup}

\author[M. Amini]{Massoud Amini}
\address{ Department of Mathematics, Faculty of Mathematical Sciences, Tarbiat Modares University, Tehran 14115--134, Iran}
\email{:mamini@modares.ac.ir }

\author[A. Bodaghi]{ Abasalt Bodaghi}
\address{Department of Mathematics, Garmsar Branch, Islamic Azad University, Garmsar, Iran}
\email{abasalt.bodaghi@gmail.com}

\author[R. Rezavand]{Reza Rezavand}
\address{School of Mathematics, Statistics and Computer Science, College of Science, University of Tehran, Enghelab Ave., Tehran, Iran}
\email{rezavand@ut.ac.ir}


\begin{document}

\maketitle

\begin{abstract}
For an inverse semigroup $S$ with the set of idempotents
$E$ and a minimal idempotent, we find necessary and sufficient conditions for the Fourier algebra $A(S)$ to be module
amenable, module character amenable, module (operator) biflat, or module (operator) biprojective.

\vspace{.3cm}
\end{abstract}

\section{Introduction}\label{s1}
 The concept of module amenability for a class of Banach
algebras which is in fact a generalization of the classical amenability has been developed by the first author in \cite{A}. Indeed, he defined the module amenability of a Banach algebra $\mathcal A$ to the case that there is an extra ${\mathfrak A}$-module structure on $\mathcal A$ and
 showed that for every inverse semigroup $S$ with subsemigroup $E$ of idempotents, the $\ell^1(E)$-module amenability of $\ell^1(S)$ is equivalent to the amenability of $S$. Later this notion of amenability is generalized by module homomorphisms in \cite{bodaghi}. Motivated by $\phi$-amenability and character amenability which are studied in \cite{kan1} and \cite{mon}, the first and second author \cite{bod1} introduced the concept of module $(\phi,\varphi)$-amenability for Banach algebras  and investigated a module character amenable
Banach algebra. They characterized the module $(\phi,\varphi)$-amenability of a Banach algebra $\mathcal A$ through vanishing of the first Hochschild module cohomology group $\mathcal H^1_{\mathfrak A}(\mathcal A, X^*)$ for certain Banach $\mathcal A$-bimodules $X$. They also proved that $\ell^1(S)$ is module character amenable (as $\ell^1(E)$-module) if and only if $S$ is amenable; for generalized notions of module character amenability refer to \cite{bel}.

The first and third authors defined the Fourier algebra $A(S)$ of an inverse
semigroup $S$ as the predual of semigroup von Neumann algebra
$L(S)$ in \cite{AR} and showed that the co-multiplication on
$L(S)$ is implemented by a multiplicative partial isometry $W$ in
$\mathcal B(\ell^2(S \times S))$. The co-algebra structure of
$L(S)$ induces a canonical algebra structure on $A(S)$, making it
a completely contractive Banach algebra. Since the set of
idempotents $E$ of $S$ naturally act on $S$ (by multiplication),
$A(S)$ has an extra structure of a Banach module on the semigroup
algebra $\ell^1(E)$. It is  shown in \cite{AR} that if $S$ is an amenable inverse semigroup $S$ with the set of
idempotents $E$ and a minimal idempotent, then the Fourier
algebra $A(S)$ is module operator amenable, as a completely
contractive Banach algebra and an operator module over
$\ell^1(E)$.

There are some other concepts related to the notions of module amenability such as module biprojectivity and module biflatness, introduced by the first and second author in \cite{boa}. In other words, these concepts are module versions of biprojectivity and biflatness for Banach algebras which are introduced by Helemskii \cite{hel}.  For every inverse semigroup $S$ with subsemigroup $E$ of idempotents, in \cite{boa}, the authors found the necessary and sufficient conditions for the  $\ell^1(S)$ to be module biprojective and module biflat (as $\ell^{1}(E)$-module). Module biflatness for the second dual of Banach algebras is studied in \cite{bj}.

In this paper we give necessary and sufficient conditions for the Fourier algebra $A(S)$ to be module
amenable, module character amenable, module (operator) biflat, or module (operator) biprojective.

%%%%%%%%%%%%%%%%%%%%%%%%%%%%%%%%%%%%%%%%%%%%%%%%%%%%%%%%%%%%%%%%%%%%%%%%%%%%%%%%%%%%%%%%%%%%%%%%%%%%%%%%%%%%%%%%%%%%%%%%%%%%%%%%%%%%%%%%%%%%%%%%%%%%%%%%%%%%%%%555
\section{Notations and Preliminary Results}\label{s1}

\subsection{Module and operator structure} In this subsection we deal with a completely contractive
Banach algebra $A$ which is a Banach module over another Banach algebra
$\mathfrak A$ with compatible actions. Our reference for operator
spaces is \cite{ER}. If $E$ and $F$ are operator spaces, and $T:
E\longrightarrow F$ is a linear map, and $T^{(n)}: M_n(E)\longrightarrow M_n(F)$ is the
$n$-th amplification of $T$ to a linear map on the corresponding
matrix algebras, then $T$ is called completely bounded if
$\|T\|_{cb} := \sup_{n\in\mathbb N} \|T^{(n)}\|< \infty$. When
$\|T\|_{cb}\leq 1$, $T$ is called a complete contraction. A
completely contractive Banach algebra $A$ is a Banach algebra which
is also an operator space such that multiplication map is a
complete contraction from $A\hat\otimes A$ to $A$ \cite{Rua},
where $A\hat\otimes A$ is the operator projective tensor product of
$A$ by itself.

Let $\mathfrak A$ be a Banach algebra and $A$ be a completely
contractive Banach algebra and a Banach $\mathfrak A$-module with
compatible actions,
\begin{equation*}
\alpha\cdot (ab)=(\alpha\cdot a)b
\,\,\,(a,b\in A,\alpha\in \mathfrak A),
\end{equation*}
and the same for the right action, then we say that $A$ is an {\it
Banach $\mathfrak A$-module}. Note that, by assumption,
\begin{equation*}
(\alpha\beta)\cdot a=\alpha\cdot(\beta\cdot a)\quad
\,\,\,(a\in A,\alpha,\beta\in \mathfrak A).
\end{equation*}
We know that
$A\hat{\otimes}_{\mathfrak A}A\cong (A\hat{\otimes}A)/I$ where $I$ is
the closed ideal generated by elements of the form $a\cdot\alpha
\otimes b- a\otimes \alpha\cdot b$, for $\alpha\in \mathfrak
A,\,a,b\in A$. This is an operator space which inherits its
operator space structure from $A\hat{\otimes}A$ \cite[Proposition
3.1.1]{ER}. We define $\omega:A\hat{\otimes}A\longrightarrow A$
by $\omega(a\otimes b) = ab$, and  $\tilde{\omega}:
A\hat{\otimes}_{\mathfrak A} A\cong(A\hat{\otimes}A)/I \longrightarrow
A/J$ by
\begin{equation*}
\tilde{\omega}(a\otimes b + I)= ab + J\,\,\,\,\,\ (a,b\in A),
\end{equation*}
both extended by linearity and continuity where
$J=\overline{\langle \omega(I)\rangle}$ is the closed ideal of
$A$ generated by $\omega(I)$. Then $\tilde{\omega},
\tilde{\omega}^{**}$ are $A$-$\mathfrak A$-module homomorphisms
\cite{A}, and since $\omega$ is a complete contraction, so is
$\tilde{\omega}$.

Let $V$ be a Banach $A$-module and a Banach $\mathfrak A$-module
with compatible actions,
\begin{equation*}
\alpha\cdot(a\cdot x)=(\alpha\cdot a)\cdot x,\,\ (a\cdot\alpha)\cdot x=a\cdot(\alpha\cdot x)
\end{equation*}
\begin{equation*}
(\alpha\cdot x)\cdot a= \alpha\cdot(x\cdot a) \,\,\,\ (a\in A, \alpha \in \mathfrak A,
x\in V),
\end{equation*}
and the same for the right or two-sided actions, such that the
module actions of $A$ on $V$ are completely bounded, then $V$ is
called an {\it operator $A$-$\mathfrak A$-module}. If moreover
\begin{equation*}
\alpha\cdot x=x\cdot\alpha \,\,\,\,\,\ (\alpha \in \mathfrak A, x\in V),
\end{equation*}
then $V$ is called a {\it commutative} $\mathfrak A$-module. A left trivial action of $\mathfrak A$ on $A$ is defined as
$\alpha\cdot a= f(\alpha)a$ for $\alpha\in \mathfrak A, a\in A$, where
$f$ is a (continuous) character of $\mathfrak A$.

Given an operator $A$-$\mathfrak A$-module $V$, a bounded map
$D:A\longrightarrow V$ is called a {\it module derivation} if
\begin{equation*}
D(a\pm b)= D(a) \pm D(b),\,\,\,\,\ D(ab)= D(a)\cdot b+ a\cdot D(b) \,\,\,\,\
(a,b\in A),
\end{equation*}
and
\begin{equation*}
D(\alpha\cdot a)= \alpha\cdot D(a),\,\,\,\,\,\ D(a\cdot\alpha)= D(a)\cdot\alpha
\,\,\,\,\ (\alpha\in \mathfrak A, a\in A).
\end{equation*}
Note that $D$ is not assumed to be $\mathbb{C}$-linear and that
$D:A\longrightarrow V$ is bounded if there exist $M>0$ such that
$\|D(a)\|\leq M\|a\|$, for each $a\in A$ \cite{A}. Let
\begin{equation*}
\|D\|= \sup_{a\neq 0}\|D(a)\|/\|a\|,
\end{equation*}
then for $D^{(n)}: M_{n}(A)\longrightarrow M_{n}(V)$, we have
$\|D^{(n)}\|\leq n^{2}\|D\|$, hence $D^{(n)}$ is bounded for each
$n$. If $\sup \|D^{(n)}\|<\infty$, we say that $D$ is {\it completely
bounded}.

A module  derivation $D$ is called {\it inner} if there
exists $v\in V$ such that
\begin{equation*}
D(a)= a\cdot v-v\cdot a \,\,\,\,\,(a\in A).
\end{equation*}

The Banach algebra $A$ is called {\it module operator amenable} (as an $\mathfrak
A$-module) if for every commutative operator $A$-$\mathfrak
A$-module $V$, each completely bounded module derivation $D:A\longrightarrow
V^{*}$ is inner.
A bounded net $\{\tilde{u_{i}}\}$ in $A\hat{\otimes}_{\mathfrak A}A$
is called a {\it module operator approximate diagonal} if
$\tilde{\omega}(\tilde{u_{i}})$ is a bounded approximate identity of
$A/J$ and
\begin{equation*}
\lim\parallel\tilde{u_{i}}\cdot a-a\cdot\tilde{u_{i}}\parallel=0\,\,\,\,\
(a\in A).
\end{equation*}
An element $\tilde{M}\in (A\hat{\otimes}_{\mathfrak A}A)^{**}$ is
called a {\it module operator virtual diagonal} if
\begin{equation*}
\tilde{\omega}^{**}(\tilde{M})\cdot a=\tilde{a}, \,\,\,\,\,\
\tilde{M}\cdot a=a\cdot\tilde{M} \,\,\,\,\,\ (a\in A),
\end{equation*}
where $\tilde{a}=a+ J^{\bot\bot}$
(see \cite{A, AR} for relation to module amenability, and \cite{A2} for a correction).

%%%%%%%%%%%%%%%%%%%%%%%%%%%%%%%%%%%%%%%%%%%%%%%%%%%%%%%%%%%%%%%%%%%%%%%%%%%%%%%%%%%%%%%%%%%%%%%%%%%%%%%%%%%%%%%%%5

\subsection{Module character amenability} Let $\mathfrak A$ be a Banach algebra with character space
$\Phi_{\mathfrak A}$ and let $\mathcal A$ be a Banach $\mathfrak
A$-bimodule with compatible actions. Let
$\varphi\in\Phi_{\mathfrak A} \cup\{0\}$ and consider the set $\Omega_{\mathcal A}$ of linear maps
$\phi: \mathcal A\longrightarrow\mathfrak A$ such that
\begin{equation*}\phi(ab)=\phi(a)\phi(b),\quad \phi(a\cdot\alpha)=\phi(\alpha\cdot
a)= \varphi(\alpha)\phi(a)\quad(a\in \mathcal A,\,\alpha\in
\mathfrak A)\end{equation*}
A bounded linear functional $m: \mathcal A^*\longrightarrow \mathbb{C}$ is
called a {\it module $(\phi,\varphi)$-mean} on $\mathcal A^*$ if
$m(f\cdot a)=\varphi\circ\phi(a)m(f)$,
$m(f\cdot\alpha)=\varphi(\alpha)m(f)$ and $m(\varphi\circ\phi)=1$
for each $f\in \mathcal A^*, a\in \mathcal A$ and $\alpha\in
\mathfrak A$. We say $\mathcal A$ is {\it module
$(\phi,\varphi)$-amenable} if there exists a module
$(\phi,\varphi)$-mean on $\mathcal A^*$. Also $\mathcal A$ is
called {\it module character amenable} if it is module
$(\phi,\varphi)$-amenable for each $\phi\in \Omega_{\mathcal A}$
and $\varphi\in\Phi_{\mathfrak A} \cup\{0\}$. Note that if $\mathfrak A=\mathbb{C}$ and $\varphi$ is the identity
map then the module $(\phi,\varphi)$-amenability coincides with
$\phi$-amenability \cite{kan1}.

It is shown in \cite[Corollary 2.4]{bod1} that module character amenability of $\mathcal A$ implies module character amenability of $\mathcal
A/J.$ We restate Remark 2.5 of \cite{bod1} for the sake of completeness as follows: Let $\phi\in \Omega_{\mathcal A}, \varphi\in\Phi_{\mathfrak
A}$ and let $\mathcal A$ be module $(\phi,\varphi)$-amenable.
Clearly $\phi((a\cdot\alpha)b-a(\alpha\cdot b))=0$, hence $\phi=0$
on $J$ and $\phi$ lifts to $\tilde \phi: \mathcal A/J\longrightarrow \mathfrak
A$ and clearly $P:=\varphi\circ\tilde \phi$ is a character of $
\mathcal A/J$. Since $\mathcal A$ is module (right) character
amenable there is $m\in  \mathcal A^{**}$ such that $m(f\cdot
a)=\varphi\circ\phi(a)m(f)$. Let $M$ be the restriction of $m$ to
$J^\perp$ then for each $F\in J^\perp$, the functional
$f(a)=F(a+J)$ is well defined and $m(f)=M(F)$ and $\langle f\cdot
a,b\rangle=\langle F\cdot(a+J),(b+J)\rangle$ and
$M(F\cdot(a+J))=m(f\cdot a)=\varphi\circ\phi(a)m(f)=P(a+J)M(F)$.
This shows that $\mathcal A/J$ is $(\tilde
\phi\circ\varphi)$-amenable. If every character $P$ of $\mathcal
A/J$ could be also constructed as above, this argument shows that
module character amenability of $\mathcal A$ implies
character amenability of $\mathcal A/J$.

Recall that a left Banach $\mathcal A$-module $X$ is called {\it{left essential}} if the linear span of $\mathcal A \cdot X=\{a\cdot x : a\in \mathcal A, \, x\in X\}$ is dense in $X$. Right essential $\mathcal A$-modules and two-sided essential $\mathcal A$-bimodules are defined similarly.

We remark that if $\mathcal A$ is a left (right) essential
$\mathfrak{A}$-module, then every $\mathfrak A$-module derivation is also
a derivation, in fact, it is linear. For if $a\in \mathcal A$, there is a sequence $(F_n)\subseteq \mathfrak A\cdot \mathcal A$ such that
$\lim_n F_n=a$. Assume that $F_n=\sum_{m=1}^{T_n}
\alpha_{n,m}\cdot a_{n,m}$ for some finite sequences
$(\alpha_{n,m})_{m=1}^{m=T_n}\subseteq \mathfrak{A}$ and
$(a_{n,m})_{m=1}^{m=T_n}\subseteq \mathcal A$. Then for each $\lambda\in
\mathbb{C}$,
\begin{eqnarray*}
D(\lambda F_n) \!\! \! & = \!\! \! & D\Big(\lambda \sum_{m=1}^{T_n} \alpha_{n,m} \cdot a_{n,m}\Big)= \sum_{m=1}^{T_n} D\big((\lambda\alpha_{n,m}) \cdot a_{n,m}\big) \vspace{0.2cm} \\ & = \!\! \! & \sum_{m=1}^{T_n} (\lambda\alpha_{n,m}) \cdot D(a_{n,m}) =  \sum_{m=1}^{T_n} \lambda D(\alpha_{n,m} \cdot a_{n,m})=  \lambda D(F_n),
\end{eqnarray*}
and so, by the continuity of $D$, $D(\lambda a)=\lambda D(a)$, if
we assume that $\mathcal A$ is a left essential $\mathfrak{A}$-module. Therefore, we have the following result.

\begin{thm}\label{char} Let $\mathcal A$ be a left \emph{(}right\emph{)} Banach essential $\mathfrak{A}$-module. Then, $\mathcal A$ is module character amenable if and only if $\mathcal A/J$ is character amenable.
\end{thm}
\begin{proof}
We include the proof for the case of left $\mathfrak{A}$-module. The other case is similar. First note that by the above discussion, the module character amenability of $\mathcal A$ implies the character amenability of $\mathcal A/J$. For the converse, let $\phi\in \Omega_{\mathcal A}$ and $\varphi\in\Phi_{\mathfrak
A}$. Assume that  $X$ is a Banach $\mathcal A$-module and Banach
${\mathfrak A}$-module such that $a\cdot x=\phi(a)\cdot x$ and
$\alpha\cdot x=x\cdot\alpha=\varphi(\alpha)x$ for all $x\in X,
a\in \mathcal A$ and $\alpha\in \mathfrak A$. Also, $D: \mathcal A \longrightarrow X^*$ is
a module derivation. Since $X$ is a commutative Banach ${\mathcal A}$-${\mathfrak A}$-module, the following module
actions are well-defined
 $$(a+J)\cdot x:=\varphi\circ\phi(a)x, \hspace{0.2cm}x\cdot (a+J):=x\cdot a \hspace{0.3cm} (x \in X ,a \in
\mathcal A),$$
and so $X$ is a Banach $\mathcal A/J$-module. Consider $\tilde{D}: \mathcal
A/J\longrightarrow X $ defined by $\tilde{D}(a+J)=D(a)$, for all $a \in \mathcal A$. Then,  $\tilde{D}$ is well-defined and $\mathbb{C}$-linear,  thus $D$ is an inner module derivation. Therefore, $\mathcal A$ is module character amenable by \cite[Theorem 2.1]{bod1}.
\end{proof}

%%%%%%%%%%%%%%%%%%%%%%%%%%%%%%%%%%%%%%%%%%%%%%%%%%%%%%%%%%%%%%%%%%%%%%%%%%%%%%%%%%%%%%%%%%%%%%%%%%%%%%%%%%%%%%%%%%%%%%%%%%%%%%%%%%%

\subsection{Fourier algebra of an inverse semigroup}
In this subsection we define the Fourier algebra $A(S)$ of an inverse
semigroup $S$ and show that it is an operator $\ell^1(E)$-module,
where $E$ is the set of idempotents of $S$ and $\ell^1(E)$ is the
semigroup algebra of $E$ \cite{AR}.

A discrete semigroup $S$ is called an inverse semigroup if for
each $x\in S$ there is a unique element $x^{*}\in S$ such that
$xx^{*}x=x$ and $x^{*}xx^{*}=x^{*}$. An element $e\in S$ is called
an idempotent if $e=e^{*}=e^{2}$. The set of idempotents of $S$ is
denoted by $E$. This is a commutative subsemigroup of $S$ with a canonical partial order: for $e,f\in E$, $e\leq f$ means that $ef=fe=e$.

For the rest of the paper, $S$ is an inverse semigroup with set of
idempotents $E$. Recall that  the {\it left regular
representation} of $S$ is the map $\lambda:S\longrightarrow
B(\ell^{2}(S))$ defined by
\begin{equation*}
\lambda (s)f(t)=
\begin{cases}
f(s^{*}t),\,\,\,ss^{*} \geq tt^{*}\\
0,\,\,\,\,\,\,\,\,\,\,\,\,\,\,\,\text{otherwise},
\end{cases}
\end{equation*}
and that $\lambda$ is a faithful $*$-representation of $S$
\cite[Proposition 2.1]{B}. We define the {\it fundamental
operator} $W:\ell^{2}(S\times S)\longrightarrow \ell^{2}(S\times
S)$ by
\begin{equation*}
W\psi(s,t)=
\begin{cases}
\psi(s,s^{*}t),\,\,\,ss^{*}\geq tt^{*}\\
0,\,\,\,\,\,\,\,\,\,\,\,\,\,\,\,\,\,\,\, \text{otherwise},
\end{cases}
\end{equation*}
for $\psi\in \ell^2(S\times S)$ and $s,t\in S$. It is easy to see that
$W$ is a bounded linear operator. It is shown in \cite{AR} that
 $W$ is a {\it multiplicative partial isometry} in the sense
of \cite{BS}.

The double commutant $L(S)=(\lambda (S))^{''} \subseteq
B(\ell^{2}(S))$ is called the (left) {\it semigroup von Neumann algebra}
of $S$. It is a bi-algebra with the co-multiplication map $\Gamma: L(S)\longrightarrow L(S)\hat\otimes L(S); \ \Gamma(\lambda(s))=\lambda(s)\otimes\lambda(s)$ (c.f. \cite{AR}). The (unique) predual $A(S)$ of $L(S)$ is a Banach space.

Let \,\,\,$\omega_{s,t}(\varphi)=\langle
\varphi\delta_{s}\mid\delta_{t}\rangle$, for each $\varphi\in
L(S)$. Then $A(S)$ is generated by the set $\{\omega_{s,t}: s,t \in
S\}$ as a Banach space \cite{AR}. One might consider an element $\omega\in A(S)$ as a function on
$S$ via
\begin{equation*}
\omega(s)= \omega(\lambda (s))\quad (s\in S).
\end{equation*}
Since $L(S)$ is a norm closed subspace of $B(\ell^{2}(S))$, it is
an operator space and hence $A(S)$ inherits a canonical operator
space structure from $L(S)$ as its predual \cite{ER}. On the
other hand, the multiplication map $\Gamma_{\ast}:A(S \times
S)\longrightarrow A(S)$ is a complete contraction (see [2,
display 3.8]), and $A(S\times S)\cong A(S)\hat{\otimes}A(S)$
\cite[7.2.4]{ER}. Therefore, $A(S)$ is a completely contractive
Banach algebra \cite{Rua}.

%%%%%%%%%%%%%%%%%%%%%%%%%%%%%%%%%%%%%%%%%%%%%%%%%%%%%%%%%%%%%%%%%%%%%%%%%%%%%%%%%%%%%%%%%%%%%%%%%%%%%%%%%%%%%%%%%%%%%%%%%%%%%%%%%%

\subsection{Module structure of A(S)} \label{m}
It is shown in \cite{AR} that the pointwise
multiplication $\omega_{s,t}\omega_{u,v}$ identifies  with the
algebra multiplication $\omega_{s,t}\cdot\omega_{u,v}$ defined by $$\omega_{s,t}\cdot\omega_{u,v}(\varphi) =
(\omega_{s,t}\otimes\omega_{u,v} )(W(\varphi\otimes 1)W^{*})\quad
(\varphi\in L(S)).$$
 We write
$\varepsilon_{a}:=\omega_{a^{*}a, a}\ .$
 The span of
$\varepsilon_{a}$'s is  dense in $A(S)$. These elements are
natural replacements for the point indicator functions
$\delta_{a}$ in $A(G)$ for a discrete group $G$. One important
difference between $\varepsilon_{a}$ and $\delta_{a}$ is that
$\varepsilon_{a}$ may have a large support \cite{AR}.

The left regular representation on $S$ lifts to the
$*$-representation $\tilde{\lambda}:\ell^1(S)\rightarrow
B(\ell^2(S))$ which is faithful by \cite[Theorem 2]{W}, and we may assume
that $\ell^1(S)\subseteq L(S)$. Since
$\tilde{\lambda}(\ell^{1}(S)){''}\supseteqq
\lambda(S)^{''}=L(S)$, $\ell^{1}(S)$ is a $w^{*}$-dense subset of
$L(S)$ \cite{AR}.

Note that $A(S)$ is an operator $\ell^{1}(E)$-module
under the actions
\begin{equation}\label{ee}
\delta_{e}.\varepsilon_{a}= \varepsilon_{a}\, ,\,\,\,\,
\varepsilon_{a}.\delta_{e}=
\begin{cases}
\varepsilon_{a},\,\,\,\,\,\,\ a^{*}a \leq e \\
0,\,\,\,\,\,\,\,\,\,\,\,\,\,\ \text{otherwise}.
\end{cases}
%\quad \left(e,f\in E, a\in S \right).
\end{equation}
If $a\in S, e\in E$, then
$\varepsilon_{a}\varepsilon_{ae}= \varepsilon_{a}$. When
$a^{*}a\leq e$, we have $\varepsilon_{a}=\varepsilon_{a}$. Also these module actions are continuous \cite{AR}.

%%%%%%%%%%%%%%%%%%%%%%%%%%%%%%%%%%%%%%%%%%%%%%%%%%%%%%%%%%%%%%%%%%%%%%%%%%%%%%%%%%%%%%%%%%%%%%%%%%%%%%%%%%%%%%%%%%%%%%%%%%%%%%%%%%%%%%%%%%%%%%%%%%%%

\subsection{Index of subsemigroups}

For an inverse semigroup $S$, the equivalence relation defined by $s\sim t$ if and only if there is $e\in E$ with $es=et$ gives the maximal group homomorphic image $G_S:=\{[s]: s\in S\}$ as the set of equivalence classes \cite{mn}. A subset $T\subseteq S$ is called an subsemigroup of $S$ if it is an $*$-semigroup in the operations of $S$. In this case, $G_T=\{[t]: t\in T\}$ is a subgroup of $S$.

For the rest of this subsection, let $T$ be a subsemigroup of $S$. We say that two elements $x,y\in S$ are $T$-equivalent, and write $x\sim_T y$ if there is an element $t\in T$ with $xy^*\sim t$.

\begin{lem} $\sim_T$ is an equivalence relation.
\end{lem}
\begin{proof}
If $x,y,z\in S$, then for each $t\in T$, $xx^*\sim tt^*\in T$. Also if $xy^*\sim t\in T$ then $yx^*\sim t^*\in T$. Finally,  if $xy^*\sim t\in T$ and $yz^*\sim s\in T$, then $$xz^*=xx^*xz^*\sim xy^*yz^*\sim ts\in T.$$
This means that $x\sim_T z$.
\end{proof}

We denote the equivalence class of $x\in S$ under the equivalence relation $\sim_T$ by $xT$ and the set of all such classes by $S/T$, and define the $E$-index of $T$ in $S$ by $$[S:T]_E:=\big|S/T\big|,$$
where the right hand side is the cardinality of $S/T$.

\begin{lem} $[S:T]_E=[G_S:G_T]$.
\end{lem}
\begin{proof}
For $x\in S$, let $[x]G_T$ be the left coset of the subgroup $G_T$, then we show that $xT\mapsto [x]G_T$ is a one-one correspondence between $S/T$ and $G_S/G_T$. This well defined, since if $x\sim_T y$ then $xy^*\sim t\in T$, hence $[x][y]^{-1}=[xy^*]=[t]\in G_T$, and so $[x]G_T=[y]G_T$. A reverse argument as above shows that the map is also injective. The surjectivity is trivial.
\end{proof}

We say that $T$ is $E$-\emph{abelian} if $st\sim ts$, for each $s,t\in T$. In this case, the subgroup $G_T\leq G_S$ is abelian. We say that $S$ is \emph{almost $E$-abelian}, if it has a subsemigroup $T$ of finite $E$-index which is $E$-abelian. This is in parallel with the notion of almost abelian groups (a group with an abelian subgroup of finite index). By the above two lemmas, we have the following result.

\begin{prop} \label{aa} $S$ is almost $E$-abelian if and only if $G_S$ is almost abelian.
\end{prop}

%%%%%%%%%%%%%%%%%%%%%%%%%%%%%%%%%%%%%%%%%%%%%%%%%%%%%%%%%%%%%%%%%%%%%%%%%%%%%%%%%%%%%%%%%%%%%%%%%%%%%%%%%%%%5

\section{Module cohomological properties of the Fourier algebra}
In this section we study module amenability, module character amenability, module (operator) biflatness, and module (operator) biprojectivity of the Fourier algebra of an inverse semigroup $S$. Let $I$ and $J$ be the corresponding closed ideals of
$\mathcal A \hat \otimes \mathcal A$ and $\mathcal A$, respectively.  We firstly bring two following definitions from \cite{boa}.

\begin{defn}\label{def1} A Banach algebra $\mathcal A$ is called {\it module
biprojective} (as an $\mathfrak A$-module) if $\widetilde{\omega}$  has a bounded right inverse which is an
 ${\mathcal A}/J$-${\mathfrak A}$-module homomorphism.
\end{defn}

\begin{defn}\label{def2} A Banach algebra $\mathcal A$ is called {\it module biflat}
(as an $\mathfrak A$-module) if $ \widetilde{\omega}^* $ has a
bounded left inverse which is an ${\mathcal A}/J$-${\mathfrak
A}$-module homomorphism.
\end{defn}

Recall that a (\emph{bounded}) \emph{left approximate identity} in a Banach algebra $\mathcal A$ is a (bounded) net $\{e_{l}\}_{l\in \mathcal L}$ in $\mathcal A$ such that $\lim_{l}e_{l}a=a$ for all $a\in \mathcal A$. Similarly, a (bounded) right approximate identity can be defined in $A$. A (bounded) approximate identity in $\mathcal A$ is both a (bounded) left approximate identity and a (bounded) right approximate identity.

%%%%%%%%%%%%%%%%%%%%%%%%%%%%%%%%%%%%%%%%%%%%%%%%%%%%%%%%%%%%%%%%%%%%%%%%%%%%%%%%%%%%%%%%%%%%%%%%%%%%%%%%%%%%%%%%%%%%%%%%%%%%%%%%%%%%%%%%%%%%%%%%%%%%%%%%%%

We say the Banach algebra $\mathfrak A$ acts trivially on $\mathcal A$ from left (right) if there is a continuous linear functional $f$ on $\mathfrak A$ such that $\alpha\cdot a=f(\alpha)a$ $(a\cdot\alpha=f(\alpha)a)$ for all $\alpha\in \mathfrak A $ and $a\in \mathcal A$.

The following lemma is proved in \cite[Lemma 3.13]{bab}.

\begin{lem}\label{tp}
If ${\mathfrak A}$ acts on $\mathcal A$ trivially from the left or right and
$\mathcal A/J$ has a right bounded approximate identity, then for
each $\alpha \in {\mathfrak A}$ and $a\in\mathcal A$ we have
$f(\alpha)a-a\cdot\alpha \in J$.
\end{lem}

The following result is main key for the main result  of this paper, coming in the next subsection.

\begin{thm}\label{p}Suppose that ${\mathfrak A}$ acts trivially on $\mathcal A$
from the left and $\mathcal A$ has a bounded approximate identity. Then
\begin{enumerate}
\item[(i)] $\mathcal A$ is module biprojective  if and only if $\mathcal A/J$ is biprojective;
\item[(ii)] $\mathcal A$ is module biflat if and only if $\mathcal A/J$ is biflat.
\end{enumerate}
\end{thm}

\begin{proof}
(i) Suppose that $\mathcal A/J$ is biprojective such that $\rho$ is
 the bounded right inverse of $\omega_{\mathcal A/J}$ which is an ${\mathcal A}/J$-module homomorphism. Define the map $\phi:(\mathcal A/J)\widehat{\otimes}(\mathcal A/J)\longrightarrow (\mathcal A\widehat{\otimes} \mathcal A)/I\cong \mathcal A\widehat{\otimes}_{\mathfrak A}\mathcal A$ by $\phi((a+J)\otimes(b+J)):=(a\otimes b)+I$. Assume that  $\{e_{j}\}$ is a bounded approximate identity for $\mathcal A$. For each $a,b,c\in \mathcal A$ and $\alpha\in \mathfrak A$, we obtain
\begin{eqnarray*}
[(a\cdot\alpha)b-a(\alpha\cdot b)]\otimes c&=&(a\cdot\alpha)b\otimes c-a(\alpha\cdot b)\otimes c\\
&=&\lim_{j}[((a\cdot\alpha)b\otimes ce_{j})-(a(\alpha\cdot b)\otimes e_{j}c)]\\
&=&\lim_{j}[((a\cdot\alpha)\otimes c)(b\otimes e_{j})-(a\otimes e_{j})((\alpha\cdot b)\otimes c)]\\
&=&\lim_{j}[((a\cdot\alpha)\otimes c)(b\otimes e_{j})-(a\otimes(\alpha\cdot c))(b\otimes e_{j})\\
&&+(a\otimes(\alpha\cdot c))(b\otimes e_{j})-(a\otimes e_{j})((\alpha\cdot b)\otimes c)\\
&&+(a\otimes e_{j})(b\otimes(\alpha\cdot c))-(a\otimes e_{j})(b\otimes(\alpha\cdot c))]\\
&=&\lim_{j}[((a\cdot\alpha)\otimes c-a\otimes(\alpha\cdot c))(b\otimes e_{j})\\
&&+(a\otimes(\alpha\cdot c))(b\otimes e_{j})-(a\otimes e_{j})((\alpha\cdot b)\otimes c\\
&&-b\otimes(\alpha\cdot c))-(a\otimes e_{j})(b\otimes(\alpha\cdot c))]\\
&=&\lim_{j}[((a\cdot\alpha)\otimes c-a\otimes(\alpha\cdot c))(b\otimes e_{j})\\
&&+(ab\otimes(\alpha\cdot c)e_{j})-(a\otimes e_{j})(f(\alpha)b\otimes c\\
&&-b\otimes f(\alpha)c)-(ab\otimes e_{j}(\alpha\cdot c))]\\
&=&\lim_{j}[((a\cdot\alpha)\otimes c-a\otimes(\alpha\cdot c))(b\otimes e_{j})\\
&&+(ab\otimes(\alpha\cdot c)e_{j})-(a\otimes e_{j})f(\alpha)(b\otimes c-b\otimes c)\\
&&-(ab\otimes e_{j}(\alpha\cdot c))]\\
&=&\lim_{j}[((a\cdot\alpha)\otimes c-a\otimes(\alpha\cdot c))(b\otimes e_{j})\in I.
\end{eqnarray*}
Similarly, $c\otimes[(a\cdot\alpha)b-a(\alpha\cdot b)]\in I$. Hence, $\phi$ is well-defined. And $\phi$ is also a module homomorphism. It is easily verified that $\widetilde{\omega}_{\mathcal A}\circ\phi=\omega_{\mathcal A/J}$. Since $\rho$ is  the bounded right inverse of $\omega_{\mathcal A/J}$, the mapping $\phi\circ\rho$ is a  bounded right inverse of $\widetilde{\omega}_{\mathcal A}$.

Conversely, assume that $\mathcal A$ is module biprojective. Since $\mathcal A$ has a bounded approximate identity, it follows from Lemma \ref{tp} that $\mathcal A/J$ is a commutative $\mathfrak A$-module. Therefore, the result follows from \cite[Proposition 2.3]{boa}.

(ii)  Suppose that $\mathcal A/J$ is biflat and $\widetilde{\rho}$ is
 the bounded left inverse of $\omega^*_{\mathcal A/J}$ which is an ${\mathcal A}/J$-module homomorphism. If the mapping $\phi$ is as in the part (i), it follows from the proof of previous part that $\phi^*\circ\widetilde{\omega}^*_{\mathcal A}=\omega^*_{\mathcal A/J}$. This means that $\widetilde{\rho}\circ\phi^*$ is the left inverse of $\widetilde{\omega}^*_{\mathcal A}$ which is an ${\mathcal A}/J$-${\mathfrak A}$-module homomorphism.

Conversely, if $\mathcal A$ is module biprojective, then $\mathcal A/J$ is biflat by Lemma \ref{tp} and \cite[Proposition 2.4]{boa}. This finishes the proof.
\end{proof}

When $\mathcal A$ is a completely contractive Banach algebra and $J$ is a closed ideal, then $A/J$ has a canonical operator space structure coming from the identification $$\mathbb M_n(A/J)=\mathbb M_n(A)/\mathbb M_n(J)$$
making it into a completely contractive Banach algebra \cite[3.1, 16.1]{ER}. We say that $A$ is {\it module operator
biprojective} ({\it biflat}, respectively) if in the above definitions, the right (left, resp.) inverse of $\widetilde{\omega}$ (of $\widetilde{\omega}^*$, resp.) is completely bounded. It is not hard to see that the above argument also works in this setting and we have the following operator analog. We give a short proof for the sake of completeness. 

\begin{prop}\label{p2}Suppose that ${\mathfrak A}$ acts trivially on $\mathcal A$
from the left and $\mathcal A$ has a bounded approximate identity. Then
\begin{enumerate}
\item[(i)] $\mathcal A$ is module operator biprojective  if and only if $\mathcal A/J$ is operator biprojective;
\item[(ii)] $\mathcal A$ is module operator biflat if and only if $\mathcal A/J$ is operator biflat.
\end{enumerate}
\end{prop}
\begin{proof}
(i) If $\mathcal A/J$ is operator biprojective and $\rho$ is
 a completely bounded right inverse of $\omega_{\mathcal A/J}$ and an ${\mathcal A}/J$-module homomorphism. Define  $\phi:(\mathcal A/J)\widehat{\otimes}(\mathcal A/J)\longrightarrow (\mathcal A\widehat{\otimes} \mathcal A)/I\cong \mathcal A\widehat{\otimes}_{\mathfrak A}\mathcal A$ by $\phi((a+J)\otimes(b+J)):=(a\otimes b)+I$. As in the proof of Theorem \ref{p}, we get that  $\phi$ is well-defined and a completely bounded module homomorphism. Also $\widetilde{\omega}_{\mathcal A}\circ\phi=\omega_{\mathcal A/J}$. Since $\rho$ is  a completely  bounded right inverse of $\omega_{\mathcal A/J}$, the map $\phi\circ\rho$ is a completely  bounded right inverse of $\widetilde{\omega}_{\mathcal A}$.
The converse follows from \cite[Proposition 2.3]{boa} using a similar argument.

(ii)  If $\mathcal A/J$ is operator biflat and $\widetilde{\rho}$ is
 a completely bounded left inverse of $\omega^*_{\mathcal A/J}$ and an ${\mathcal A}/J$-module homomorphism, then for the map $\phi$  as in  part (i),  $\phi^*\circ\widetilde{\omega}^*_{\mathcal A}=\omega^*_{\mathcal A/J}$. This means that $\widetilde{\rho}\circ\phi^*$ is a completely bounded left inverse of $\widetilde{\omega}^*_{\mathcal A}$ and an ${\mathcal A}/J$-${\mathfrak A}$-module homomorphism. The converse follows from a modification of Lemma \ref{tp} and \cite[Proposition 2.4]{boa}. 
\end{proof}

\begin{rem} \label{rem}
The condition that $\mathcal A$ has a bounded approximate identity (in Theorem \ref{p} and Proposition \ref{p2}) happens to be a strong condition for the Fourier algebra. Indeed the Fourier algebra of a group has a bounded approximate identity iff the group is amenable \cite{Le}. It is not known when the Fourier algebra of an inverse semigroup has a bounded approximate identity. However, from the proof it follows that the above theorem also holds under the weaker condition that  $\mathcal A=\mathcal A^{2}$ (this is used in the argument that shows the constructed map $\phi$ is well defined.) As we shall see in the proof of the next theorem, this condition is automatically satisfied for the Fourier algebra $A(S)$ of an inverse semigroup $S$.
\end{rem}

In the next theorem which is our main result, we consider $A(S)$ as an operator $\ell^1(E)$-module (see subsection \ref{m}.)

\begin{thm}\label{p3} Let $S$ be an inverse semigroup.

\begin{enumerate}
\item[(i)] $A(S)$ is module amenable if and only if $S$ is almost abelian.
\item[(ii)]  When $S$ is left amenable, then $A(S)$ is module biflat or module biprojective if and only if $S$ is almost abelian.
\item[(iii)]  $A(S)$ is always operator biflat and operator biprojective.
\item[(iv)]  If $E$ has a minimum element, then $A(S)$ is module operator amenable if and only if $S$ is amenable. If $E$ has no minimum
element, then $A(S)$ is always module operator amenable.
\item[(v)]  If $E$ has a minimum element, then $A(S)$ is module character amenable if and only if $S$ is amenable. If $E$ has no minimum
element, then $A(S)$ is always module character amenable.
\end{enumerate}
\end{thm}
\begin{proof}
First note that by \cite[Proposition 4.13]{AR} the condition $\mathcal A=\mathcal A^{2}$ holds for $\mathcal A=A(S)$ and Remark \ref{rem} shows that Theorem \ref{p} and Proposition \ref{p2} apply without having a bounded approximate identity.

(i) This follows from \cite[Propositions 3.2, 3.3]{abe} and \cite[Theorem 2.3]{fr}.

(ii) If $A(S)$ is module biflat, and so is $A(G_S)$ by Theorem \ref{p}. Since $G_S$ is amenable, it follows from the main theorem in \cite{run} that $G_S$ is almost abelian. Therefore $S$ is almost abelian by Proposition \ref{aa}. Conversely, if $S$ is almost abelian, then so is $G_S$ and $A(S)$ is module biprojective by the last corollary of \cite{run} and Theorem \ref{p}.

(iii) This follows from Proposition \ref{p2}, \cite[Theorem 4.5]{wo}, and \cite[Theorem 2.4]{ar} (note that a discrete group is always a [QSIN]-group, in the sense of  \cite[page 374]{ar}.

(iv) This is \cite[Theorem 5.14]{AR}.

(v) Since $A(S)$ is always an essential $\ell^1(E)$-module with actions (\ref{ee}), the result follows from Theorem \ref{char} and \cite[Corollary 2.4]{mon}.
\end{proof}

%\section*{Acknowledgements}

%The authors would like to thank the reviewer for careful reading of the paper, giving some useful comments and suggestions. 

\bibliographystyle{line}
\bibliography{JAMS-paper}

\end{document}